\documentclass[3p]{article}

\usepackage{graphicx}
\usepackage{amsmath}
\usepackage{amsfonts}
\usepackage{amssymb}
 \usepackage{amscd}

\newtheorem{theorem}{Theorem}[section]

\newtheorem{corollary}[theorem]{Corollary}

\newtheorem{lemma}[theorem]{Lemma}

\newenvironment{proof}[1][Proof]{\textbf{#1.} }{\ \rule{0.5em}{0.5em}}

\title{On a new  closed formula for  the solution of second order linear difference equations and applications}
\author{Issam Kaddoura \ \ \ \ \ \ \\{\small Department of Mathematics,  Faculty  of Arts and Science, Lebanese International
University, Saida, Lebanon} \and Bassam Mourad\\{\small Department of Mathematics,  Faculty of Science, Lebanese University, Beirut, Lebanon}}

\begin{document}

\maketitle

\begin{abstract} In this note, we establish a new closed formula for  the solution of homogeneous second order linear difference equations with constant coefficients by using matrix theory. This in turn gives new closed formulas concerning all sequences of this type such as the Fibonacci and Lucas sequences.  As applications; we show that Binet's formula in this case is valid for negative integers as well. Then,  we  find new summation formulas relating the elements of such sequences. Finally,  Maple simulations that show the computational advantages of our methods against Binet's formula are presented.
\end{abstract}

\paragraph*{keywords.}{\footnotesize   difference equations; Fibonacci sequence; Lucas sequence;
Tchebychev polynomials; Pell sequence; Generalized Binet's formula}
\paragraph*{AMS.}  {\footnotesize 11B39, 65Q30}

\section{Introduction}
For any $(n,z)\in \mathbb{Z}\times \mathbb{C}$ where $\mathbb{Z}$ is the usual ring of
integers and $\mathbb{C}$ is the field of complex numbers,  we call a
\textit{generalized linear second order recurrent sequence} any sequence which is given by the following second order linear difference equation:
\begin{eqnarray} R_{n+1}(z)=f(z)R_n(z)+g(z)R_{n-1}(z), \ R_0(z)=h(z), \ R_1(z)=k(z)\end{eqnarray}  where $f,g,$ $h$ and $k$
are any complex functions\footnote{\scriptsize{Note that  $f,g,$ $h$ and $k$ can also be considered as multi-variables functions, however this obviously will not
 affect the results.}}. Without loss of generality, we may assume that
$f(z)\neq0$ and $g(z)\neq0$
since otherwise we obtain one trivial case for $f(z)=g(z)=0,$ and two other cases which are easy to
handle; one corresponds to $f(z)=0, g(z)\neq 0$ and the other is
associated with $f(z)\neq 0, g(z)=0.$ In addition, for convenience,
we may sometimes use the notation: $R_{n}=R_{n}(z)=R(n,f,g,h,k)$ to refer to the nth term of such a sequence, and we omit the argument from the functional notation when there is no ambiguity; so $f$,  for example,  will mean $f(z)$. Our
intention here is to study this sequence assuming that
$f,$ $g,$ $h$ and $k$ are any complex functions.

Certainly there is an extensive work in the literature concerning linear second order difference equations
 and their applications (see for example \cite{ca,kos,kosh}). However only a few deals with finding a general framework under which all these sequences come together under one theory (see for example \cite{yt}). Our main objective in this note is to deal with this issue subject to the extra condition that $n$ is any integer in $\mathbb{Z}$.

This paper is organized as follows. In Section 2,
we present a new closed formula solution for the
sequence defined by (1) which in turn gives new closed formulas for many  sequences
of this type such as Fibonacci, Lucas, Pell, Pell-Lucas, Jacobsthal and Jacobsthal-Lucas number sequences as well as Tchebychev,
Fibonacci and Lucas polynomials.
 Section 3 deals with showing that Binet's formula in this case is valid for negative integers as well.
Finally, new summation formulas relating the elements of the sequence $\{R_n\}_n$ are presented in Section 4. Finally,  Maple simulations that show the computational advantages of our methods are presented.

\section{Main Results}
We shall start by fixing some notation. The determinant of a square matrix $X$ will be denoted by $|X|$. For our purposes, we next introduce the following four matrices which will be used throughout this paper and they constitute the essential foundations of our main results. Define
$$A=\left[
\begin{array}
[c]{ccc}
2gh+f^{2}h-fk & \ \ \ \  2kg-fhg\\

2k-fh & \ \ \ 2gh+fk
\end{array}
\right]  , \ \ \ \  B=\left[
\begin{array}
[c]{ccc}%
0 &  g\\
1 & f
\end{array}
\right]  , $$ $C=\left[
\begin{array}
[c]{ccc}%
2g & fg\\
f & \ \  f^{2}+2g
\end{array}
\right]  $ and $D=\left[
\begin{array}
[c]{ccc}%
-f &  2g\\
2 & f
\end{array}
\right]  $ where $f$ $g,$ $h$
and $k$ are as above. Now a simple check shows that the matrices  $A$, $B$, $C$ and $D$ are pairwise commuting.

As a result,  we have the following useful lemma that relates the matrices $A$, $B$, $C$ and $D$ with only  $R_n$ and $R_{n-1}$ so that it gives the interesting consequence that solving
a second order linear homogeneous difference equation that has constant coefficients and which is coupled with initial conditions, is essentially equivalent to solving a non-homogeneous linear first order difference equation.

\begin{lemma}
Let the matrices  $A$, $B$, $C$ and $D$ be defined as above and let $R_n$ be the general term of the sequence defined by (1). Then for any integer $n,$ we have
\begin{eqnarray}
AB^{n}=CR_{n}+DgR_{n-1}.
\end{eqnarray}

\end{lemma}

\begin{proof}
Using induction, We first do the proof for the case $n\geq1$. It is clear that (2) is true for $n=1$ since  $AB=CR_1+DgR_0$ as an inspection  shows that
$$
\left[
\begin{array}
[c]{ccc}
2gh+f^{2}h-fk & \ \ \  2kg- fhg\\
2k-fh & \ \ \ 2gh+fk
\end{array}
\right]  \left[
\begin{array}
[c]{ccc}
0 &  g\\
1 &   f
\end{array}
\right]  =
\left[
\begin{array}
[c]{ccc}
2g &  fg\\
f & \ \  f^{2}+2g
\end{array}
\right]  k+\left[
\begin{array}
[c]{ccc}
-f &  2g\\
2 &   f
\end{array}
\right]  gh .
$$
Now suppose that (2) is true for $n,$  then we have
$$
\left[
\begin{array}
[c]{ccc}
2gh+f^{2}h-fk & \ \ \  2kg- fhg\\
2k-fh & \ \ \  2gh+fk
\end{array}
\right]  \left[
\begin{array}
[c]{cc}
0 &  g\\
1 &  f
\end{array}
\right]  ^{n} =
\left[
\begin{array}
[c]{cc}%
2g &  fg\\
f & \ \ \ f^{2}+2g
\end{array}
\right]  R_{n}+\left[
\begin{array}
[c]{cc}
-f &  2g\\
2 &  f
\end{array}
\right]  gR_{n-1}.
$$
Multiplying both sides of this last equation to the right by $B=\left[
\begin{array}
[c]{cc}
0 &  g\\
1 &   f
\end{array}
\right]  ,$ we obtain the following equality:
\begin{align*} AB^{n+1}&=\left(  CR_{n}+DgR_{n-1}\right)  B\\
&=\left[
\begin{array}
[c]{cc}
fgR_{n}+2g^{2}R_{n-1} & \ \ g(2g+f^{2})R_{n}+fg^{2}R_{n-1}\\
  &   \\
(f^{2}+2g)R_{n}+fgR_{n-1} & \ \ \ \ f(f^{2}+3g)R_{n}+g(2g+f^{2})R_{n-1}
\end{array}
\right]  \\
&=\left[
\begin{array}
[c]{cc}
2g &  fg\\
  &   \\
f & \ \ f^{2}+2g
\end{array}
\right]  (fR_{n}+gR_{n-1})+\left[
\begin{array}
[c]{cc}
-f & 2g\\
  &   \\
2 &   f
\end{array}
\right]  gR_{n}\\ &=CR_{n+1}+DgR_{n}.
\end{align*}
On the other hand, for the case $n=0,$ it can be  easily seen that the  sequence defined by (1) becomes $R_{-1}=g^{-1}(k-fh).$
Therefore, we have $A=Ch+D(k-fh)$ so that (2) is also valid for $n=0.$  Moreover, as
$B^{-1}=\left[ \begin {array}{cc} -{\frac {f}{g}} & 1\\ \noalign{\medskip}\frac{1}{g} & 0\end {array} \right],$ then a simple check shows that $AB^{-1}=CR_{-1}+DgR_{-2}$ so that (2) is also true for $n=-1$ as well.
Finally, the
proof for the case $n<-1,$ can also be done by induction in a similar
manner as in the first case, and the proof is complete.
\end{proof}

Next note that an inspection  shows that the following determinant formulas are  valid:
$$ |A|=\left(  f^{2}+4g\right)  \left(
gh^{2}-k^{2}+fhk\right)  ,$$ and $$\left|CR_{n}+DgR_{n-1}\right |=g\left( f^{2}+ 4g\right)  \left(  R_{n}^{2}-gR_{n-1}
^{2}-fR_{n}R_{n-1}\right)  .$$ As a conclusion, we have
the following lemma which also appears as Theorem 5 in \cite{yt} for $z$ real, albeit arrived at by different means.

\begin{lemma}
Let $R_n,$ $f$, $g$, $h$ and $k$ be as given above.  If $f^{2}+4g\neq0$ then we have:
\begin{equation}
R_{n}^{2}-gR_{n-1}^{2}-fR_{n}R_{n-1}=\left(
-gh^{2}+k^{2}-fhk\right)  (-g)^{n-1},
\end{equation}
for all integers $n.$
\end{lemma}

\begin{proof}
Taking determinants of both sides of (2)  we obtain $|AB^{n}|=|CR_{n}+DgR_{n-1}|.$  More explicitly, 
$$
\left|  \left[
\begin{array}
[c]{cc}
2gh+f^{2}h-fk &\ \ \  2gk- fgh\\
2k-fh & \ 2gh+fk
\end{array}
\right]  \left[
\begin{array}
[c]{cc}
0 &  g\\
1 & f
\end{array}
\right]  ^{n}\right|  \ =
\left|  \left[
\begin{array}
[c]{cc}
2g &   fg\\
f &  \ \ f^{2}+2g
\end{array}
\right]  R_{n}+\left[
\begin{array}
[c]{cc}
-f &  2g\\
2 &   f
\end{array}
\right]  gR_{n-1}\right|  .
$$
Using the fact that $|XY|=|X||Y|$ for any square matrices of the same size, we get
$$\left(  f^{2}+4g\right)  \left(  gh^{2}-k^{2}+fhk\right)  (-g)^{n}=\\
g\left(
f^{2}+4g\right)  \left(  R_{n}^{2}
-gR_{n-1}^{2}-fR_{n}R_{n-1}\right)  .
$$
Therefore (3) is valid.
\end{proof}

As a consequence, we are now able to present the following result.

\begin{theorem}
For any integer $n$ and for any complex functions $f\neq 0$ and $g\neq 0$, the general term $R_n$ of the recurrence  $R_{n+1}=fR_{n}+gR_{n-1}
,$ with seeds $R_0=h,R_1=k,$
satisfies one of the following:
\newline 1) If $f^{2}+4g\neq0$ and
$gh^{2}-k^{2}+fhk\neq0,$ then
\begin{equation} R_n^{2}= \frac{(-g)^{n}(gh^{2}-k^{2}+fhk)}{\left(f^{2}+4g\right)} \left| M+(-g)^nB^{-2n}\right | \end{equation}
where
 $M= \left[
\begin{array}
[c]{cc}
\frac{gh^{2}+f^{2}h^{2}+k^{2}-2fhk}{gh^{2}-k^{2}+fhk} & \ \ \ \  gh\frac{2k-fh}{gh^{2}-k^{2}+fhk}\\
  &   \\
h\frac{2k-fh}{gh^{2}-k^{2}+fhk} & \  \frac{gh^{2}+k^{2}}{gh^{2}-k^{2}+fhk}
\end{array}
\right], $ and $B=\left[
\begin{array}
[c]{ccc}
0 &  g\\
1 & f
\end{array}
\right]. $

2) If $f^{2}+4g=0$ then we obtain $R_{n+1}\ =fR_{n}-\dfrac{f^{2}}
{4}R_{n-1}$ which implies that
\begin{eqnarray}
R_{n}=\frac{n(2k-fh  )+fh}{2^{n}}  f^{n-1}.
\end{eqnarray}
If in addition, $gh^{2}-k^{2}+fhk=0,$ then $R_{n}=h\left(  \frac{f}{2}\right)  ^{n}.$
\newline 3) If $f^{2}+4g\neq0$ and
$gh^{2}-k^{2}+fhk=0,$  then we obtain the following closed formula: $R_{n}=\dfrac{k^{n}}{h^{n-1}}.$
\end{theorem}
\begin{proof}
\underline{Case 1}: First notice that $-gB^{-2}=\left[
\begin{array}
[c]{cc}
-\frac{1}{g}\left(  g+f^{2}\right)  &  f\\
\dfrac{f}{g} &  -1
\end{array}
\right]  ,$ and  $MB=BM$. Therefore,  the right-hand side of (4) can be  rewritten as
\begin{align*} \text{RHS} &= \frac{(-g)^{n}(gh^{2}-k^{2}+fhk)}{\left(f^{2}+4g\right)} \left| B^{-n}\right | \left |MB^n+(-g)^nB^{-n}\right |\\
  &= \frac{gh^{2}-k^{2}+fhk}{\left(f^{2}+4g\right)} \left| B^{-n}\right | \left |MB^n+(-g)^nB^{-n}\right |.
  \end{align*}
 Now from equality (2), we know that $B^{n}=A^{-1}\left(  CR_{n}+DgR_{n-1}\right)  .$ Finding $A^{-1}$ and multiplying, we obtain
 $$B^{n}=
\left[
\begin{array}
[c]{cc}
g\frac{hR_{n}-kR_{n-1}}{gh^{2}-k^{2}+fhk} &   g\frac{-kR_{n}+fhR_{n}+ghR_{n-1}
}{gh^{2}-k^{2}+fhk}\\
  &   \\
\frac{-kR_{n}+fhR_{n}+ghR_{n-1}}{gh^{2}-k^{2}+fhk} & \ \ \ \  \frac{ghR_{n}
-fkR_{n}+f^{2}hR_{n}-gkR_{n-1}+fghR_{n-1}}{gh^{2}-k^{2}+fhk}
\end{array}
\right]  .$$

As a result, we get
\begin{align*}
 M B^{n} & =\left[
\begin{array}
[c]{cc}
g\frac{hR_{n}+kR_{n-1}-fhR_{n-1}}{gh^{2}-k^{2}+fhk} & g\frac{kR_{n}+ghR_{n-1}
}{gh^{2}-k^{2}+fhk}\\
  &   \\
\frac{kR_{n}+ghR_{n-1}}{gh^{2}-k^{2}+fhk} &  \frac{ghR_{n}+fkR_{n}+gkR_{n-1}
}{gh^{2}-k^{2}+fhk}
\end{array}\right] \\
&=
\frac{1}{gh^{2}-k^{2}+fhk}\left[
\begin{array}
[c]{cc}
g\left(  hR_{n}+kR_{n-1}-fhR_{n-1}\right)  & g\left(  kR_{n}+ghR_{n-1}\right)\\
  &   \\
kR_{n}+ghR_{n-1} & ghR_{n}+fkR_{n}+gkR_{n-1}
\end{array}
\right]  .
\end{align*}
Similarly, we can write
\begin{align*} (-g)^{n}B^{-n} &= (-g)^{n}\left(  CR_{n}+DgR_{n-1}\right)^{-1} A\\
&=\left(  -g\right)  ^{n}\left[
\begin{array}
[c]{cc}
\frac{ghR_{n}-ftR_{n}+f^{2}hR_{n}-gtR_{n-1}+fghR_{n-1}
}{g\left(  -gR_{n-1}^{2}+R_{n}^{2}-fR_{n}R_{n-1}\right)  } & \frac{-kR_{n}+fhR_{n}+ghR_{n-1}}{gR_{n-1}^{2}-R_{n}^{2}+fR_{n}R_{n-1}}\\
  &   \\
\frac{-tR_{n}+fhR_{n}+ghR_{n-1}}{g\left(  gR_{n-1}
^{2}-R_{n}^{2}+fR_{n}R_{n-1}\right)  } \ \  \ \ \ \ \ \ & \frac
{hR_{n}-kR_{n-1}}{-gR_{n-1}^{2}+R_{n}^{2}-fR_{n}R_{n-1}}
\end{array}
\right] \\
 &= \tfrac
{g}{\left(  gh^{2}-k^{2}+fhk\right)  }\left[
\begin{array}
[c]{cc}
\frac{(gh-fk+f^{2}h)R_{n}-(gk-fgh)R_{n-1}}{g} &   (k-fh)R_{n}-ghR_{n-1}\\
  &   \\
\frac{(k-fh)R_{n}-ghR_{n-1}}{g} \ \ \ \ \  & hR_{n}-kR_{n-1}
\end{array}
\right]  . \end{align*} Substituting these two expressions of $B^{n}$ and $(-g)^{n}B^{-n}$,
 we obtain an expression that involves a determinant of sum of two matrices. More explicitly,
\begin{multline*}
\text{RHS}=\tfrac{(gh^{2}-k^{2}+fhk)}{\left( f^{2}+4g\right)  } \left|  \left[
\begin{array}
[c]{cc}
g\left(  hR_{n}+kR_{n-1}-fhR_{n-1}\right)  & g\left(  kR_{n}+ghR_{n-1}\right)\\
  &   \\
kR_{n}+ghR_{n-1} & ghR_{n}+fkR_{n}+gkR_{n-1}%
\end{array}
\right]  \right. \\
\left.  + g\left[
\begin{array}[c]{cc}
\frac{ghR_{n}-fkR_{n}+f^{2}hR_{n}-gkR_{n-1}+fghR_{n-1}}{g} &  \ \ \ \ \ \  \tiny{kR_{n}-fhR_{n}-ghR_{n-1}}\\
  &   \\
\frac{kR_{n}-fhR_{n}-ghR_{n-1}}{g} \ \ \ \ \  \ &  \tiny{hR_{n}-kR_{n-1}}
\end{array}
\right]  \right|  .
\end{multline*}
After arranging terms and simplifying by $gh^{2}-k^{2}+fhk$, we obtain
\begin{align*} \text{RHS}&=\frac{1}{\left( f^{2}+4g\right)(gh^{2}-k^{2}+fhk)}\left|
\left[
\begin{array}
[c]{cc}
R_{n}\left(  f^{2}h+2gh-fk\right)  & \ \  gR_{n}\left(  2k-fh\right) \\
  &   \\
R_{n}\left(  2k-fh\right)  & \ \ \  R_{n}\left(  2gh+fk\right)
\end{array}
\right]  \right|\\  &=R_{n}^{2}.\end{align*} This completes the proof of the first
case.
\newline \underline{Case 2}:   (5) can be easily obtained by a direct substitution.
Now if $gh^{2}-k^{2}+fhk=0,$ then solving for $k$ we obtain
$k=\frac{fh}{2} $. Substituting in (5), then the proof can be easily completed.
\newline\underline{ Case 3}: Observe first that $gh^{2}-k^{2}+fhk=0,$ implies that $k=h\dfrac{f\pm\sqrt{f^{2}+4g}}{2}$.  Moreover, (3) implies that $$R_{n}^{2}-gR_{n-1}
^{2}-fR_{n}R_{n-1}=\left(  -gh^{2}+k^{2}-fhk\right)  (-g)^{n-1}=0.$$ As a result, we obtain
$R_{n}=\dfrac{f\pm\sqrt{f^{2}+4g}}{2}R_{n-1}$, and hence $R_{n}=\dfrac{k}{h}R_{n-1}$. Thus, $R_{n}=\dfrac{k^{n}}{h^{n-1}}.$
\end{proof}

An immediate consequence of the preceding theorem is the following corollary
which gives new closed formulas for the following sequences; see for example \cite{ho,kos,kosh,na} and the references therein for more details on such sequences.

\begin{corollary}
The following new closed formulas hold.
\newline(1) If $f=x, g=1, h=0$ and $k=1,$ we get the
Fibonacci polynomials:  $$ f_{n}(x)=\sqrt{\frac{(-1)^{n+1}%
}{4+x^{2}} \left|  \left[
\begin{array}
[c]{cc}%
-1 &   0\\
0 &  -1
\end{array}
\right]  +\left[
\begin{array}
[c]{cc}%
-(1+x^{2}) &  x\\
x &  -1
\end{array}
\right]  ^{n}\right|  }.$$
\newline(2) For $f=1,$
$g=1,$ $h=0$ and $k=1,$ then the Fibonacci numbers are given by: $$ F_{n}=\sqrt{\frac{(-1)^{n+1}}%
{5} \left|  \left[
\begin{array}
[c]{cc}
-1 &  0\\
0 &  -1
\end{array}
\right]  +\left[
\begin{array}
[c]{cc}
-2 &  1\\
1 &   -1
\end{array}
\right]  ^{n}\right|  }.$$
\newline(3) If $f=x, g=1, h=2,$ and
$k=x$  we get the Lucas polynomials:
$$ l_{n}(x)=\sqrt{(-1)^{n}\left|  \left[
\begin{array}
[c]{cc}
1 &  0\\
0 &   1
\end{array}
\right]  +\left[
\begin{array}
[c]{cc}
-(1+x^{2}) &   x\\
x & -1
\end{array}
\right]  ^{n}\right|  }.$$
\newline(4) For $f=1,$ $g=1,$ $h=2$ and $k=1,$ the Lucas
numbers are given by: $$L_{n}=\sqrt{(-1)^{n}\left|  \left[
\begin{array}
[c]{cc}
1 &  0\\
0 &  1
\end{array}
\right]  +\left[
\begin{array}
[c]{cc}
-2 &  1\\
1 &  -1
\end{array}
\right]  ^{n}\right|  } . $$
\newline(5) If $f=1, g=2x, h=1$ and $k=1,$ we get the
Jacobsthal polynomials: $$J_{n}(x)=\sqrt{\frac{(-1)^{n}(2x)^{n+1}}{1+8x}%
\left|  \left[
\begin{array}
[c]{cc}
1 &  1\\
\frac{1}{2x} &  1+\frac{1}{2x}
\end{array}
\right]  +\left[
\begin{array}
[c]{cc}
-\frac{1}{2x}-1 &  1\\
\frac{1}{2x} &  -1
\end{array}
\right]  ^{n}\right|  }.$$
\newline(6) If $f=1, g=2, h=0$ and $k=1,$ we get the
Jacobsthal numbers: $$J_{n}=\sqrt{\frac{(-2)^{n}}{-9}%
\left|  \left[
\begin{array}
[c]{cc}
-1 &  0\\
0 &  -1
\end{array}
\right]  +\left[
\begin{array}
[c]{cc}
-3/2 &  1\\
1/2 &  -1
\end{array}
\right]  ^{n}\right|  }.$$
\newline(7) For $f=1, g=2x, h=2$ and $k=1,$ we obtain  the
Jacobsthal-Lucas polynomials: $$ j_{n}(x)=\sqrt{(-2x)^{n}
\left|  \left[
\begin{array}
[c]{cc}
1 &  0\\
0 &  1
\end{array}
\right]  +\left[
\begin{array}
[c]{cc}
-\frac{1}{2x}-1 &  1\\
\frac{1}{2x} &  -1
\end{array}
\right]  ^{n}\right|  }.$$
\newline(8) For $f=1, g=2, h=2$ and $k=1,$ we obtain  the
Jacobsthal-Lucas numbers: $$ j_{n}=\sqrt{(-2)^{n}
\left|  \left[
\begin{array}
[c]{cc}
1 &  0\\
0 &  1
\end{array}
\right]  +\left[
\begin{array}
[c]{cc}
-3/2 &  1\\
1/2 &  -1
\end{array}
\right]  ^{n}\right|  }.$$
\newline(9) If $f=2,g=1,h=2$ and $k=2,$ we get the Pell-Lucas numbers:
 $$ Q_{n}=\sqrt{(-1)^{n}\left|
\left[
\begin{array}
[c]{cc}
1&  0\\
0 &   1
\end{array}
\right]  +\left[
\begin{array}
[c]{cc}
-5 &  2\\
2 &  -1
\end{array}
\right]  ^{n}\right|  }.$$
\newline(10)  For $f=2,$ $g=1,$ $h=0$ and $k=1,$ we obtain the
Pell numbers: $$ P_n=\sqrt{\frac{(-1)^{n+1}}%
{8}\left|  \left[
\begin{array}
[c]{cc}
-1 &  0\\
0 &  -1
\end{array}
\right]  +\left[
\begin{array}
[c]{cc}
-5 &  2\\
2 & -1
\end{array}
\right]  ^{n}\right|  }.$$
\newline(11) For $f=2x, g=-1, h=1$ and $k=x,$ we obtain the
Tchebychev polynomial of the first kind: $$ T_{n}%
(x)=\sqrt{\frac{1}{4}\left|  \left[
\begin{array}
[c]{cc}
1 &  0\\
0 &  1
\end{array}
\right]  +\left[
\begin{array}
[c]{cc}
4x^{2}-1 &  2x\\
-2x &  -1
\end{array}
\right]  ^{n}\right |  }.$$

\end{corollary}

\section{Generalized Binet formula}

The main objective in this section is to show that Binet's formula for the difference equation (1) is also valid for negative integers as well. But first, we need to introduce some notation. Let $P$ be  the matrix which is defined by
$$P:=\left[ \begin {array}{cc} 1&1\\ \noalign{\medskip}{\frac {f+
\,\sqrt {{f}^{2}+4\,g}}{2g}}&{\frac {f-\sqrt {{f}^{2}+4\,g}}{
2g}}\end {array} \right].$$ Now it is easy to check that  the inverse $P^{-1}$ of $P$ is given by
$$ P^{-1}= \left[ \begin {array}{cc} {\frac {-f+\sqrt {{f}^{2}+4\,g}}{2\sqrt {{f}^{2}+4\,g}}}&{\frac {g}{\sqrt {{f}^{2}+4\,g}}}
\\ \noalign{\medskip}{\frac {f+\sqrt {{f}^{2}+4\,g}}{2\sqrt {{f}^{
2}+4\,g}}}&-{\frac {g}{\sqrt {{f}^{2}+4\,g}}}\end {array} \right].
$$  In addition,  $P$ diagonalizes all the matrices $A$, $B$, $C$ and $D$. More explicitly, an inspection shows that
\begin{itemize}
\item[(i)] $P^{-1}BP=\left[ \begin {array}{cc} \frac{1}{2}\left(f+\sqrt {{f}^{2}+4\,g}\right)&0\\ \noalign{\medskip}0
&\frac{1}{2}\left(f-\sqrt {{f}^{2}+4\,g}\right)\end {array} \right],$

\item[(ii)] $P^{-1}DP=\left[ \begin {array}{cc} \sqrt {{f}^{2}+4\,g}&0\\ \noalign{\medskip}0
&-\sqrt {{f}^{2}+4\,g}\end {array} \right]$,
\item[(iii)] $P^{-1}CP=\left[ \begin {array}{cc} \frac{1}{2}\left(f^2+4g+\sqrt {{f}^{2}+4\,g}\right)&0\\ \noalign{\medskip}0
&\frac{1}{2}\left(f^2+4g-\sqrt {{f}^{2}+4\,g}\right)\end {array} \right],$
\item[(iv)] $P^{-1}AP=\left[ \begin {array}{cc} \frac{1}{2}\left(h(f^2+4g)+(2k-fh)\sqrt {{f}^{2}+4\,g}\right)&0\\ \noalign{\medskip}0
&\frac{1}{2}\left(h(f^2+4g)-(2k-fh)\sqrt {{f}^{2}+4\,g}\right)\end {array} \right].$
\end{itemize}

As a result, we have the following conclusion which ensures the validity  of Binet's formula  for negative integers as well.

\begin{corollary} For any integer $n$, it holds that
$$ R_n=\left(\frac{h}{2}+\frac{2k-fh}{2\sqrt{f^2+4\,g}}\right)\left(\frac{f+\sqrt{f^2+4\,g}}{2}\right)^n +
\left(\frac{h}{2}-\frac{2k-fh}{2\sqrt{f^2+4\,g}}\right)\left(\frac{f-\sqrt{f^2+4\,g}}{2}\right)^n.$$
\end{corollary}
\begin{proof}
From (2), we can easily get that $P^{-1}APP^{-1}B^{n}P=P^{-1}CPR_{n}+P^{-1}DPgR_{n-1}$. After substitution and equating terms in the preceding matrix equation, we obtain the following system: \\
{\scriptsize $\left\{
\begin{array}
[c]{c}
\frac{1}{2}\left(h(f^2+4g)+(2k-fh)\sqrt{f^2+4\,g}\right)\left(\frac{f+\sqrt{f^2+4\,g}}{2}\right)^n=\frac{R_n}{2}\left(f^2+4g+f\sqrt{f^2+4\,g}\right)
+\frac{gR_{n-1}}{2}\sqrt{f^2+4\,g}\\
\frac{1}{2}\left(h(f^2+4g)-(2k-fh)\sqrt{f^2+4\,g}\right)\left(\frac{f-\sqrt{f^2+4\,g}}{2}\right)^n=\frac{R_n}{2}\left(f^2+4g-f\sqrt{f^2+4\,g}\right)
-\frac{gR_{n-1}}{2}\sqrt{f^2+4\,g}.
\end{array}
\right  .$ } \\

The proof can be easily achieved by adding the two equations in the preceding system.
\end{proof}

\section{  New relations satisfied by the elements of sequence (1)}

In this section, we present new summation
relations that the elements $R_{i}$ of sequence (1) satisfy.

We start with the following result which can be proven by using only the definition of sequence (1).
\begin{theorem} For any integer $n$ and for any complex functions
$f,g,h$ and $k,$  the following two equations, concerning the elements $R_i$ of sequence (1),  are valid:\\ \noindent
\hspace{1cm}  $  \left(  f^{2}-g^2+2g-1\right) \sum\limits_{i=1}^{n}R_{i}^{2}=(R_{n+1}
^{2}-k^{2})-g^{2}(R_{n}^{2}-h^{2})+2Sg,  $
 $$
f( f^{2}-g^2+2g-1)\sum\limits_{i=1}^{n}R_{i}R_{i-1}
=(1-g)(R_{n+1}^{2}-k^{2})+g(g-1+f^{2})(R_{n}^{2}-h^{2})+S(1-f^{2}-g^{2}), \ \ \ \ \ \ \ \ \
$$
where
$ S=\left\{
\begin{array}
[c]{c}
\left(  k^{2}-gh^{2}-fhk\right)  \dfrac{1-(-g)^{n}}{1+g}\text{ \ \ \ if
\ \ \ }g\neq-1\\
n\left(  k^{2}-gh^{2}-kht\right)  \text{ \ \ \ \ \ \ \ \ \ \ \ \ \ \  if
\ \ \ }g=-1.
\end{array}
\right.
$
\end{theorem}

\begin{proof}
It is easy to see that
\begin{eqnarray}
\sum\limits_{i=1}^{n}R_{i+1}^{2}=\sum\limits_{i=1}^{n}(fR_{i}+gR_{i-1})^{2}=f^{2}\sum_{i=1}^{n}R_{i}^{2}+g^{2}
\sum\limits_{i=1}^{n}R_{i-1}^{2}+2fg\sum\limits_{i=1}^{n}R_{i}R_{i-1}.
\end{eqnarray}
Now if we let $x_n=\sum\limits_{i=1}^{n}R_{i+1}^{2}$ and $y_n=\sum\limits_{i=1}
^{n}R_{i}R_{i-1},$ then it is easy to see that we can write $\sum\limits_{i=1}^{n}R_{i}
^{2}=x_n+k^{2}-R_{n+1}^{2}$ and $\sum\limits_{i=1}^{n}R_{i-1}^{2}=x_n+h^{2}
+k^{2}-R_{n}^{2}-R_{n+1}^{2}.$ Substituting these in (6) and then rearranging
the terms, we get
\begin{eqnarray}
x_n(1-f^{2}-g^{2})-2fgy_n=f^{2}(k^{2}-R_{n+1}^{2})+g^{2}(h^{2}+k^{2}-R_{n}
^{2}-R_{n+1}^{2}).
\end{eqnarray}
Making use of (3) of Lemma 2.2 in (6), we get $$\sum\limits_{i=1}^{n}R_{i}
^{2}-g\sum\limits_{i=1}^{n}R_{i-1}^{2}-f\sum\limits_{i=1}^{n}R_{i}R_{i-1}
=\sum\limits_{i=1}^{n}\left(  k^{2}-gh^{2}-fhk\right)  (-g)^{i-1}.$$ After
arranging terms, we obtain
\begin{eqnarray}
(1-g)x_n-fy_n=S-(k^{2}-R_{n+1}^{2})+g(h^{2}+k^{2}-R_{n}^{2}-R_{n+1}^{2}),
\end{eqnarray} where \begin{align*} S &=\sum\limits_{i=1}^{n}\left(  k^{2}-gh^{2}-fhk\right)  (-g)^{i-1}\\
&=\left\{
\begin{array}
[c]{c}
\left(  k^{2}-gh^{2}-fhk\right)  \dfrac{1-(-g)^{n}}{1+g}\text{ \ \ \ if
\ \ \ }g\neq-1\\
n\left(  k^{2}-gh^{2}-kht\right)  \text{ \ \ \ \ \ \ \ \ \ \ \ \ \ \ \  if
\ \ \ }g=-1.
\end{array}
\right.
\end{align*}   Solving the system (7) and (8) for $x_n$ and $y_n$, we then obtain two
equations where the first is in $x_n$ and is given by $$\left(  f-g+1\right)
\left(  f+g-1\right)  x_n=(f^{2}+2g-g^{2})(R_{n+1}^{2}-k^{2})-g^{2}(R_{n}
^{2}-h^{2})+2Sg,$$ and the second equation involves only $y_n$ and is given by
$$f\left(  f-g+1\right)  \left(  f+g-1\right)  y_n=(1-g)(R_{n+1}^{2}
-k^{2})+g(g-1+f^{2})(R_{n}^{2}-h^{2})+S(1-f^{2}-g^{2}).$$ This completes the proof.
\end{proof}

Another different relation involving products of  elements of sequence (1), is given in the next theorem (for similar results of this kind, see, for example,  \cite{yt}).
\begin{theorem}
For any integers $n$ and $i$ in $\mathbb{Z}$, the elements $R_{n}$ of sequence (1) satisfy the following:
\begin{multline}
(k^{2}-fkh-gh^{2})R_{i+n}+(f^{2}h-fk+gh)R_{i}R_{n}+ hg^{2}R_{i-1}R_{n-1}+(fh-k)g\left ( R_{n}R_{i-1}+R_{i}R_{n-1}\right)=0.
\end{multline}
\end{theorem}

\begin{proof}
Since $A$ and $B$ commute, then for any integers $i$ and $n$, we obtain the matrix equation $AB^nAB^i=A^2B^{n+i}$.
 By (2), 
    this last equation can be rewritten as $$  (CR_n+DgR_{n-1})(CR_i+DgR_{i-1})=A(CR_{n+i}+DgR_{n+i-1}).$$ After substitution, we obtain following two equations:
$$g\left(  4g+f^{2}\right)  \left(  R_{i}R_{n}-hR_{n+i}-kR_{n+i-1}+gR_{i-1}R_{n-1}+fhR_{n+i-1}\right)=0$$
 $$g\left(  4g+f^{2}\right)  \left(-kR_{n+i}+gR_{n}R_{i-1}+ gR_{i}R_{n-1}+fR_{i}R_{n}-ghR_{n+i-1}\right)=0.$$
   Since $4g+f^2\neq 0,$ then we conclude the following two equations:
$$gh\left(  R_{i}R_{n}-hR_{n+i}+gR_{i-1}R_{n-1}+(fh-k)R_{n+i-1}\right)  =0,$$
$$(fh-k)(-kR_{n+i}+gR_{n}R_{i-1}+gR_{i}R_{n-1}+fR_{i}R_{n}-ghR_{n+i-1})=0. $$
Finding the value of $(fh-k)R_{n+i-1}$ from the first equation and substituting it in the second equation, we obtain (9).
\end{proof}

\section*{Appendix}
\subsection*{Computational simulations}
In order to show the efficiency of our method, we shall evaluate the performance of our technique against Binet's method for the same calculation by giving  an overview on the rough computational effort of each method for particular values of $n$. More precisely, we claim that
one of the striking and perhaps surprising features of each formula in the preceding corollary is that it has a lower computational cost comparing with that  of Binet's formula.  Without loss of generality, we shall only perform simulations  for Fibonacci and Lucas numbers as it can be easily done  for the rest in a similar fashion.
First, recall that  Binet's formulas for  Fibonacci and Lucas numbers are respectively given by
$$ FB(n) :=\frac{1}{\sqrt{5}}\left[\left(\frac{1+\sqrt{5}}{2}\right)^n-\left(\frac{1-\sqrt{5}}{2}\right)^n\right],$$
$$ LB(n) :=\left[\left(\frac{1+\sqrt{5}}{2}\right)^n+\left(\frac{1-\sqrt{5}}{2}\right)^n\right].$$
We shall compare these formulas with the following formulas from the preceding corollary:
$$ FO(n):=\sqrt{\frac{(-1)^{n+1}}{5} \left|  \left[\begin{array}[c]{cc}
-1 &  0\\
0 &  -1
\end{array}
\right]  +\left[
\begin{array}
[c]{cc}
-2 &  1\\
1 &   -1
\end{array}
\right]  ^{n}\right|  }$$
 $$LO(n):=\sqrt{(-1)^{n}\left|  \left[
\begin{array}
[c]{cc}
1 &  0\\
0 &  1
\end{array}
\right]  +\left[
\begin{array}
[c]{cc}
-2 &  1\\
1 &  -1
\end{array}
\right]  ^{n}\right|  }. $$

We shall do the comparison by using Maple. In fact, we will start by
 applying the command ``GreaterComplexity'' for $n=10$ and then for $n=10^6$ that compares the complexity of Binet 's formula against our formula for Lucas numbers. Then, we go into more details by using the command ``cost''
  for $n=10, \ 100,\ 1000,  \ 100000$.  Though it will not be used here, it is worthy to note that one can also use the command ``SortByComplexity'' for the same purpose.
Indeed, the following Maple 18 session proves our claim for Fibonacci and Lucas numbers.\\
$
>with(LinearAlgebra):\\
>with(SolveTools):\\
>with(codegen, cost, optimize):$\\
$>A := Matrix([[-2, 1], [1, -1]]); M := Matrix([[1, 0], [0, 1]]);$
 $$A:=\left[
\begin{array}
[c]{cc}
-2 &  1\\
1 &  -1
\end{array}
\right]$$
$$M:= \left[
\begin{array}
[c]{cc}
1 &  0\\
0 &  1
\end{array}
\right]$$
$>GreaterComplexity\left(\left(\frac{1+5^{1/2}}{2}\right)^{10}+\left(\frac{1-5^{1/2}}{2}\right)^{10}, \left(Determinant\left(M+A^{10}\right)\right)^{1/2}\right);$
$$true$$
$>GreaterComplexity\left(\left(\frac{1+5^{1/2}}{2}\right)^{1000000}+\left(\frac{1-5^{1/2}}{2}\right)^{1000000}, \left(Determinant\left(M+A^{1000000}\right)\right)^{1/2}\right);$
$$true$$
$>\emph{\textbf{For \ n=10}}\\
>FB(10) := \frac{1}{5^{1/2}}.\left[\left(\frac{1+5^{1/2}}{2}\right)^{10}-\left(\frac{1-5^{1/2}}{2}\right)^{10}\right]; expand(\%);\\
 LB(10) := \left(\frac{1+5^{1/2}}{2}\right)^{10}+\left(\frac{1-5^{1/2}}{2}\right)^{10}; expand(\%); cost(FB(10)); cost(LB(10));$
$$ FB(10) :=\frac{1}{\sqrt{5}}\left[\left(\frac{1+\sqrt{5}}{2}\right)^{10}-\left(\frac{1-\sqrt{5}}{2}\right)^{10}\right]$$
$$55$$
$$ LB(10) :=\left(\frac{1+\sqrt{5}}{2}\right)^{10}+\left(\frac{1-\sqrt{5}}{2}\right)^{10}$$
$$123$$
$$25 \ multiplications+3 \ additions+6 \ functions$$
$$22 \ multiplications+3 \ additions+4 \ functions$$
$>FO(10) := \left(\frac{-1}{5}.Determinant\left(-M+A^{10}\right)\right)^{1/2}: simplify(\%);\\
 LO(10) := \left(Determinant\left(M+A^{10}\right)\right)^{1/2}: simplify(\%);cost(FO(10)); cost(LO(10));$
$$55$$
$$123$$
$$2 \ functions + \ multiplications $$
$$2 \ functions + \ multiplications $$
$>\textbf{\emph{For \ n=100}}\\
>FB(100) := \frac{1}{5^{1/2}}.\left[\left(\frac{1+5^{1/2}}{2}\right)^{100}-\left(\frac{1-5^{1/2}}{2}\right)^{100}\right]; expand(\%);
 LB(100) :=\\ \left(\frac{1+5^{1/2}}{2}\right)^{100}+\left(\frac{1-5^{1/2}}{2}\right)^{100}; expand(\%); cost(FB(100)); cost(LB(100));$
$$ FB(100) :=\frac{1}{\sqrt{5}}\left[\left(\frac{1+\sqrt{5}}{2}\right)^{100}-\left(\frac{1-\sqrt{5}}{2}\right)^{100}\right]$$
$$354224848179261915075$$
$$ LB(100) :=\left(\frac{1+\sqrt{5}}{2}\right)^{100}+\left(\frac{1-\sqrt{5}}{2}\right)^{100}$$
$$792070839848372253127$$
$$205 \ multiplications+3 \ additions+6 \ functions$$
$$202 \ multiplications+3 \ additions+4 \ functions$$
$>FO(100) := \left(\frac{-1}{5}.Determinant\left(-M+A^{100}\right)\right)^{1/2}: simplify(\%);
 LO(100) :=\\ \left(Determinant\left(M+A^{100}\right)\right)^{1/2}: simplify(\%); cost(FO(100)); cost(LO(100));$
$$354224848179261915075$$
$$792070839848372253127$$
$$2 \ functions + \ multiplications $$
$$2 \ functions + \ multiplications $$
$>\emph{\textbf{For \ n=1000}}\\
>FB(1000) := \frac{1}{5^{1/2}}.\left[\left(\frac{1+5^{1/2}}{2}\right)^{1000}-\left(\frac{1-5^{1/2}}{2}\right)^{1000}\right]; expand(\%);
 LB(1000) :=\\ \left(\frac{1+5^{1/2}}{2}\right)^{1000}+\left(\frac{1-5^{1/2}}{2}\right)^{1000}; expand(\%); cost(FB(1000)); cost(LB(1000));$
$$ FB(1000) :=\frac{1}{\sqrt{5}}\left[\left(\frac{1+\sqrt{5}}{2}\right)^{1000}-\left(\frac{1-\sqrt{5}}{2}\right)^{1000}\right]$$
\begin{multline*}
4346655768693745643568852767504062580256466051737178040248172908953655541794\\
9051890403879840079255169295922593080322634775209689623239873322471161642996 \ \ \\
440906533187938298969649928516003704476137795166849228875 \ \ \ \ \ \ \ \ \ \ \ \ \ \ \ \ \ \ \ \ \ \ \ \ \ \ \ \ \ \ \
\end{multline*}
$$ LB(1000) :=\left(\frac{1+\sqrt{5}}{2}\right)^{1000}+\left(\frac{1-\sqrt{5}}{2}\right)^{1000}$$
\begin{multline*}
97194177735908175207981982079326473737797879155345685082728081084772518818444\\
81526908061914904596829767957830540320934740116303690766057397174086246375180 \ \ \\
1641201490284097309096322681531675707666695323797578127  \ \ \ \ \ \ \ \ \ \ \ \ \ \ \ \ \ \ \ \ \ \ \ \ \ \ \ \ \ \ \ \ \ \ \
\end{multline*}
$$2005 \ multiplications+3 \ additions+6 \ functions$$
$$2002 \ multiplications+3 \ additions+4 \ functions$$
$>FO(1000) := \left(\frac{-1}{5}.Determinant\left(-M+A^{1000}\right)\right)^{1/2}: simplify(\%);
 LO(1000) :=\\ \left(Determinant\left(M+A^{1000}\right)\right)^{1/2}: simplify(\%); cost(FO(1000)); cost(LO(1000));$
\begin{multline*}
4346655768693745643568852767504062580256466051737178040248172908953655541794\\
9051890403879840079255169295922593080322634775209689623239873322471161642996 \ \ \\
440906533187938298969649928516003704476137795166849228875 \ \ \ \ \ \ \ \ \ \ \ \ \ \ \ \ \ \ \ \ \ \ \ \ \ \ \ \ \ \ \
\end{multline*}
\begin{multline*}
97194177735908175207981982079326473737797879155345685082728081084772518818444\\
81526908061914904596829767957830540320934740116303690766057397174086246375180 \ \ \\
1641201490284097309096322681531675707666695323797578127  \ \ \ \ \ \ \ \ \ \ \ \ \ \ \ \ \ \ \ \ \ \ \ \ \ \ \ \ \ \ \ \ \ \ \
\end{multline*}
$$2 \ functions + \ multiplications $$
$$2 \ functions + \ multiplications $$
$>\emph{\textbf{For \ n=100000}}\\
>FB(100000) := \frac{1}{5^{1/2}}.\left[\left(\frac{1+5^{1/2}}{2}\right)^{100000}-\left(\frac{1-5^{1/2}}{2}\right)^{100000}\right]; expand(\%);
 LB(100000) :=\\ \left(\frac{1+5^{1/2}}{2}\right)^{100000}+\left(\frac{1-5^{1/2}}{2}\right)^{100000}; expand(\%); cost(FB(100000)); cost(LB(100000));$
$$ FB(100000) :=\frac{1}{\sqrt{5}}\left[\left(\frac{1+\sqrt{5}}{2}\right)^{100000}-\left(\frac{1-\sqrt{5}}{2}\right)^{100000}\right]$$
\begin{multline*}
25974069347221724166155034021275915414880485386517696584724770703952534543511\\
27368626555677283671674[...20699 digits...]9259130435572321635660895603514383883939\\
018953166274355609970015699780289236362349895374653428746875 \ \ \ \ \ \ \ \ \ \ \ \ \ \ \ \ \ \ \ \ \ \ \ \ \ \
\end{multline*}
$$ LB(100000) :=\left(\frac{1+\sqrt{5}}{2}\right)^{100000}+\left(\frac{1-\sqrt{5}}{2}\right)^{100000}$$
\begin{multline*}
580797847126813635602250680046023872205978286661269934813695510400141800856342\\
5084641321277250355647[...20699 digits...]083642318157867798447687871806108202216120\\
3529764821597059303847402451379477857726676584685986328127 \ \ \ \ \ \ \ \ \ \ \ \ \ \ \ \ \ \ \ \ \ \ \ \ \ \ \ \ \ \
\end{multline*}
$$200005 \ multiplications+3 \ additions+6 \ functions$$
$$200002 \ multiplications+3 \ additions+4 \ functions$$
$>FO(100000) := \left(\frac{-1}{5}.Determinant\left(-M+A^{100000}\right)\right)^{1/2}: simplify(\%);
 LO(100000) :=\\ \left(Determinant\left(M+A^{100000}\right)\right)^{1/2}: simplify(\%); cost(FO(100000)); cost(LO(100000));$
\begin{multline*}
25974069347221724166155034021275915414880485386517696584724770703952534543511\\
27368626555677283671674[...20699 digits...]9259130435572321635660895603514383883939\\
018953166274355609970015699780289236362349895374653428746875 \ \ \ \ \ \ \ \ \ \ \ \ \ \ \ \ \ \ \ \ \ \ \ \ \ \
\end{multline*}
\begin{multline*}
580797847126813635602250680046023872205978286661269934813695510400141800856342\\
5084641321277250355647[...20699 digits...]083642318157867798447687871806108202216120\\
3529764821597059303847402451379477857726676584685986328127 \ \ \ \ \ \ \ \ \ \ \ \ \ \ \ \ \ \ \ \ \ \ \ \ \ \ \ \ \ \
\end{multline*}
$$2 \ functions + \ multiplications $$
$$2 \ functions + \ multiplications $$

\end{document}